\documentclass[english]{alggeom}

\usepackage[all,cmtip]{xy}

\usepackage{amsmath,amssymb,amscd,amsfonts}

\usepackage{babel}
\usepackage{amstext}
\usepackage{amsmath}
\usepackage{amsfonts}
\usepackage{latexsym}
\usepackage{ifthen}
\usepackage{xypic}
\xyoption{all}
\renewcommand{\contentsname}{Contents\\{\footnotesize\normalfont(A table
of contents should normally not be included)}}

\newcommand{\codim}{{\rm codim}}

\newcommand\sV{{\mathcal V}}

\newcommand\bC{{\mathbb C}}
\newcommand\bQ{{\mathbb Q}}

\newcommand\bN{{\mathbb N}}

\newcommand\can{{\rm can}}
\newcommand\Aut{{\rm Aut}}

\newcommand\bP{{\mathbb P}}

\newcommand\vol{{\rm vol}}

\def\nd{\mathop{\rm nd}\nolimits}

\def\can{\mathop{\rm can}\nolimits}

\def\Det{\mathop{\rm det}\nolimits}
\def\Div{\mathop{\rm div}\nolimits}

\def\Aut{\mathop{\rm Aut}\nolimits}

\let\ep=\varepsilon

\def\bQ{{\mathbb Q}}
\def\bC{{\mathbb C}}

\def\bN{{\mathbb N}}
\def\bP{{\mathbb P}}

\renewcommand{\theequation}{\arabic{section}.\arabic{lemma}.\arabic{equation}}
\newcounter{lemma}[section]
\renewcommand{\thelemma}{\strut\kern-3pt\arabic{section}.\arabic{lemma}}
\newtheorem{lemma1}[lemma]{\setcounter{equation}{0}}
\let\saveref=\ref
\def\ref#1{\strut\kern3pt{\saveref{#1}}}
\def\eqref#1{({\saveref{#1}})}

\newenvironment{lemma}{\begin{lemma1}{\bf Lemma.}}{\end{lemma1}}

\newenvironment{theorem}{\begin{lemma1}{\bf Theorem.}}{\end{lemma1}}
\newenvironment{proposition}{\begin{lemma1}{\bf Proposition.}}{\end{lemma1}}

\newenvironment{remark}{\begin{lemma1}{\bf Remark.}\rm}{\end{lemma1}}
\newenvironment{definition}{\begin{lemma1}{\bf Definition.}}{\end{lemma1}}

\begin{document}

\title[Kodaira dimension of algebraic fiber spaces over surfaces]{Kodaira dimension of algebraic fiber spaces over surfaces}

\author{Junyan Cao}
\email{junyan.cao@imj-prg.fr}
\address{Universit\'e Paris 6 \\
Institut de Math\'{e}matiques de Jussieu\\
4, Place Jussieu, Paris 75252, France }

\thanks{This work is partially supported by the Agence Nationale de la Recherche grant ``Convergence de Gromov-Hausdorff en g\'{e}om\'{e}trie k\"{a}hl\'{e}rienne"
(ANR-GRACK).}

\begin{abstract} 
In this short note we prove the Iitaka $C_{nm}$ conjecture for algebraic fiber spaces over surfaces. 
\end{abstract}
\maketitle

\section{Introduction}

Let $p: X\rightarrow Y$ be a fibration between two projective manifolds. A central problem in birational geometry is the 
\textit{Iitaka conjecture}, stating that
\begin{equation}
\kappa (X) \geq\kappa (Y) +\kappa (X/Y) 
\end{equation}
where $\kappa (X)$ is the Kodaira dimension of $X$ and $\kappa (X/Y)$ is the Kodaira dimension of the generic fiber.
\smallskip

\noindent In this note, we prove that the log-version of Iitaka conjecture holds true, provided that the base $\dim Y \leq 2$; this generalizes
a result obtained by C. Birkar in \cite[Thm 1.4]{Bir09} and a result of Y. Kawamata in \cite{Kaw82a}.
More precisely, we have the following statement.

\begin{theorem}\label{main}
Let $p: X\rightarrow Y$ be a fibration between two projective manifolds. Let $F$ be the generic fiber and let $\Delta$ be a $\bQ$-effective klt divisor on $X$.
Set $\Delta_F := \Delta |_F$.
If $\dim Y \leq 2$, then we have
\begin{equation}\label{maininequlity}
\kappa (K_X +\Delta) \geq \kappa (K_F +\Delta_F) +\kappa (Y) . 
\end{equation}
\end{theorem}

\bigskip

We will next explain the main steps of the proof. Since $\dim Y \leq 2$, 
we can assume that $K_Y$ is nef, by using the MMP. Three cases as follows might occur, according to the 
Kodaira dimension of $Y$.
\smallskip

If $\kappa (Y) \geq 1$, the inequality \eqref{maininequlity} is quickly verified 
by using \cite{Kaw82a}. 
\smallskip

If $\kappa (Y)=0$, by the classification theory we know that $Y$ is a torus or a K3 surface, modulo a 
finite \'{e}tale cover. If $Y$ is a torus, \eqref{maininequlity}
is proved in \cite{CH11} for the absolute case (i.e., $\Delta= 0$) 
and in \cite{CP17} for the klt pair case. 
Therefore, to 
prove \eqref{maininequlity}, it is enough to assume that $Y$ is a K3 surface. 
In order to treat this case, we rely on two
 main ingredients, namely the positivity of the direct images $p_\star (m K_{X/Y}+m\Delta)$
and the geometry of orbifold Calabi-Yau surfaces. 
Different aspects of the first topic were 
extensively studied in
\cite{Gri70, Fuj78, Kaw82a, Kaw82b, Kol85, Vie95, Ber09, BP10, PT14, BC15, Fuj16, HPS16, KP17}, among many other articles.
In our set-up, this implies 
that $\det p_\star (m K_{X/Y}+m\Delta)$ is pseudo-effective (by using \cite{PT14}).
As $Y$ is a K3 surface, the numerical dimension of $\det p_\star (m K_{X/Y}+m\Delta)$ coincides with its Iitaka dimension. 
If the numerical dimension $\nd (\det p_\star (m K_{X/Y}+m\Delta)) \geq 1$,
we achieve our goals by standard arguments. If 
$\nd (\det p_\star (m K_{X/Y}+m\Delta)) = 0$,
we can show that there exists a finite set of exceptional curves $[C_i]$ on $Y$ such that $p_\star (m K_{X/Y}+m\Delta)$ is hermitian flat on $Y\setminus (\cup C_i)$,
by using the results 
in \cite{PT14,CP17,HPS16}. 
At this point we 
use the second ingredient, namely the uniformization theorem for the compact K\"{a}hler orbifolds with trivial first Chern class, cf. \cite{Cam04a}. We thus infer that 
the fundamental group of $Y\setminus (\cup C_i)$ is almost-abelien. 
Therefore we can construct sufficient elements in $H^0 (X, m K_{X/Y} +m\Delta)$ by using parallel transport, and \eqref{maininequlity} is proved.

\medskip

\textbf{Acknowledgements:} We would like to thank J.-A. Chen who brought my attention to this problem and valuable suggestions on the article.
The author would also like to thank M. P\u{a}un for extremely enlightening discussions on the topic and several important suggestions on the first draft of the article. 
We are also indebted with H.-Y. Lin for pointing out to us the reference \cite{Cam04a}, which plays a key role in this article.
Last but not least, we would like to thank the anonymous referees for their valuable suggestions.

\section{Preparation}

\noindent In this section, we will recall the uniformization theorem for the compact K\"{a}hler orbifolds with trivial first Chern class \cite{Cam04a} (cf. also
\cite{CC14, GKP15}) as well as
the results concerning singular metrics on vector bundles and the positivity of direct images, cf. \cite{BP08,PT14,Rau15,Pau16,HPS16} for more details.

\medskip

\noindent First of all, we recall a few basic definitions concerning compact K\"{a}hler orbifolds by following \cite{Cam04a}.
\begin{definition}\label{basicdefn}\cite[Definition 3.1, Definition 5.1]{Cam04a}
\begin{enumerate}

\item[\rm{(i)}] A compact K\"{a}hler orbifold is a compact K\"{a}hler normal variety with only quotient singularities, i.e., for every point
$a\in X$, we can find a neighbourhood $U$ of $a$, and a biholomorphism $\psi: U \rightarrow \widetilde{U}/G$ where $\widetilde{U}$
is an open set in $\bC^n$ and $G\subset GL (n, \bC)$ is a finite subgroup acting on $\widetilde{U}$ with $\psi (a)= 0$. 
For every $g\in G$, the set of the fixed points of $g$ is of codimension at least $2$.

\smallskip

\item[\rm{(ii)}] Let $X^\star$ be the smooth locus of a compact K\"{a}hler orbifold $X$. We say that $X$ is simply connected
in the sense of orbifolds if $X^\star$ is simply connected.

\smallskip

\item[\rm{(iii)}] A holomorphic morphism between two K\"{a}hler orbifolds $r: X' \rightarrow X$ is said to be an orbifold cover, if 
it satisfies the following two conditions:
\begin{enumerate}

\item[$\bullet$] The restriction of $r$ to $r^{-1} (X^\star)$ is an \'{e}tale cover.

\item[$\bullet$] For every $a\in X$ with its neighbourhood $\widetilde{U}/G$ (cf. $(i)$),
each component of $r^{-1} (\widetilde{U}/G)$ is of the form $\widetilde{U} /G'$ for some subgroup $G'$ of $G$, and the restricted 
morphism $r|_{\widetilde{U} /G'}$ is nothing but the natural quotient morphism $\widetilde{U} /G' \rightarrow \widetilde{U} /G$.
\end{enumerate}

\smallskip

\item[\rm{(iv)}] A $m$-dimensional compact K\"{a}hler orbifold $X$ is called Calabi-Yau (resp. Hyperk\"{a}hler), 
if it is simply connected in the sense of orbifold (cf. $(ii)$) and it admits
a Ricci-flat K\"{a}hler metric such that the holonomy (when restricted to $X^\star$) is $SU (m)$ (resp. $Sp (m/2)$). 
\end{enumerate}

\end{definition}
\medskip

\noindent We state here the uniformization theorem for the compact K\"{a}hler orbifolds with trivial first Chern class, established in \cite{Cam04a}. The statement parallels the 
classical case of smooth  K\"{a}hler manifolds with trivial first Chern class. 

\begin{theorem}\cite[Thm 6.4]{Cam04a}\label{orbifolddecomp}
Let $X$ be a compact K\"{a}hler orbifold with $c_1 (X)=0$. Then $X$ admits a finite orbifold cover $\overline{X} = \overline{C} \times \overline{S} \times T$,
where $\overline{C}$ (resp. $\overline{S}$) is a finite product of Calabi-Yau K\"{a}hler orbifold (resp. Hyperk\"{a}hler) and $T$ is a complex torus.
\end{theorem}
\smallskip

\noindent Let $Y$ be a K3 surface and let $\cup C_i$ be a set of exceptional curves on $Y$.
By Grauert's criterion \cite[III, Thm 2.1]{BHPV}, there is a contraction morphism $\tau: Y\rightarrow Y_{\can}$ which 
contracts all $C_i$ to some points $p_i$ in a normal space $Y_{\can}$. As $K_Y$ is trivial, we know 
that $Y_{\can}$ is in fact a compact K\"{a}hler orbifold (\cite[Def 4.4, Rk 4.21]{KM98}, \cite[Example 3.2]{Cam04a}) with $c_1 (Y_{\can})=0$.
\medskip

\noindent As a corollary of Theorem \ref{orbifolddecomp}, we have.

\begin{proposition}\cite[Cor 6.7]{Cam04a}\label{uniformation}
Let $Y$ be K3 surface and let $\cup C_i$ be some exceptional curves on $Y$.
Then $\pi_1 (Y \setminus (\cup C_i))$ is almost-abelien. 

Moreover, let $\tau: Y\rightarrow Y_{\can}$ be the morphism which contracts the exceptional curves $C_i$ to some points $p_i \in Y_{\can}$.
If $\pi_1 (Y \setminus (\cup C_i))$ is not finite, there exists a finite orbifold cover from a complex torus $T$ to $Y_{\can}$:
$$\sigma: T\rightarrow Y_{\can} .$$
In particular, $\sigma$ is a non-ramified cover over $Y_{\can} \setminus (\cup p_i)$ 
and $\sigma^{-1} (p_i)$ is of codimension $2$ for every $i$.
\end{proposition}
\medskip

\noindent The following non-vanishing property for pseudo-effective line bundles on $K3$ surfaces is an immediate 
consequence of the abundance theorem (which holds true in dimension two).

\begin{proposition}\label{keyprop}
Let $Y$ be a $K3$ surface (in the smooth sense) and let $L$ be a pseudo-effective line bundle on $Y$. 
Then $L$ is $\bQ$-effective.
\end{proposition}

\begin{proof}
Since $L$ is pseudo-effective, by using Zariski decomposition for surface \cite[Thm 1.12]{Fuj79}, we know that
$$L \equiv_\bQ \sum_{i=1}^p  a_i [C_i] +M ,$$
where $a_i \in \bQ^+$, $C_i$ are negative intersection curves, $M$ is nef and $M\cdot C_i =0$ for every $i$.
Since $Y$ is $K3$, all nef line bundles on $Y$ are effective
(cf. \cite[VIII, Prop 3.7]{BHPV}).
Therefore $L$ is $\bQ$-effective.
\end{proof}

\begin{remark}\label{semiample}
It is well-known that a nef line bundle $M$ on a K3 surface is semiample. If its numerical dimension $\nd (M)$ is equal to one, 
then it induces an elliptic fibration over $\bP^1$.
\end{remark}

\medskip

\noindent 
In the second part of this section we will recall a few definitions and results about the singular metrics on vector 
bundles and the positivity of direct images. We refer to \cite{BP08,Rau15,PT14,Pau16,HPS16} for more details.
\begin{definition}\label{singularmetric}
Let $E \rightarrow X$ be a holomorphic vector bundle on a manifold $X$ (which is not necessary compact). Locally, a singular hermitian metric $h_E$ on
$E$ is a measurable map from $X$ to the space of non-negative Hermitian forms on the fibers.
We say that $(E, h_E)$ is negatively curved, if $0 <\det h_E < +\infty$ almost everywhere and
$$x\rightarrow \ln |u|_{h_E} (x)  \qquad x\in X$$
is a psh function, for any choice of a holomorphic local section $u$ of $E$.

\smallskip

We say that the pair $(E, h_E)$ is positively curved, if the dual $(E^\star , h_E ^\star)$ is negatively curved.
We note it by $i\Theta_{h_E} (E) \succeq 0$.
\end{definition}

When $h_E$ is smooth, "positively curved" is nothing but the classical Griffiths semi-positivity.
The following result proved in \cite{PT14} plays an important role in this article.

\begin{theorem}\cite[Thm 5.1.2]{PT14}\label{maintool}
Let $p: X\rightarrow Y$ be a fibration between two projective manifolds and let $L$ be a line bundle on $X$ with a possibly singular metric $h_L$
such that $i\Theta_{h_L} (L) \geq 0$. Let $m\in \mathbb{N}$ such that the multiplier ideal sheaf $\mathcal{I} (h_L ^{\frac{1}{m}} |_{X_y})$
is trivial over a generic fiber $X_y$, namely $\int_{X_y} |e_L|_{h_L} ^{\frac{2}{m}} < +\infty$, where $e_L$ is a basis of $L$. 

Let $Y_1$ be the locally free locus of $p_\star (m K_{X/Y} +L)$. 
Then the vector bundle $p_\star (m K_{X/Y} +L)$ over $Y_1$ 
admits a possibly singular hermitian metric $h$ such that $i\Theta_{h} (p_\star (m K_{X/Y} +L)) \succeq 0$ on $Y_1$. 
Moreover, $h$ induces a possibly singular metric
$\det h$ on the line bundle $\det p_\star (m K_{X/Y} + L)$ over $Y$ such that 
$$i\Theta_{\det h} (\det p_\star (m K_{X/Y} + L)) \geq 0 \qquad\text{ on }Y$$
in the sense of current. 
\end{theorem}

\begin{remark}\label{hermtimetricconstr}
Let us recall briefly the construction of the metric $h$: 
Let $h_B$ be the $m$-relative Bergman kernel metric on $K_{X/Y}+\frac{1}{m} L$ constructed in \cite[A.2]{BP10}.
Set $L_1 :=(m-1) K_{X/Y}+ L$ and $h_{L_1} := (m-1) h_B +\frac{1}{m} h_L $. Thanks to \cite[A.2]{BP10}, we know that
\begin{equation}\label{postivcur}
 i\Theta_{h_{L_1}} (L_1) \geq 0 \qquad\text{on }X
\end{equation}
in the sense of current. 
Now $h_{L_1}$ induces a Hodge type metric $h$ on $\pi_\star (m K_{X/Y} +L)$ on the smooth locus $Y_0$ of $\pi$ as follows: let $X_y$ be a smooth fiber and 
let $f\in H^0 (X_y, m K_{X/Y} +L)$. As $m K_{X/Y}+L = K_{X/Y} +L_1$, the norm 
$$\|f\|_h ^2 := \int_{X_y} |f|_{h_{L_1}} ^2$$
is well defined. Since $h_L$ is not necessarily smooth, $h$ is a possibly singular hermitian metric on $(\pi_\star (m K_{X/Y} +L), Y_0)$.
Thanks to \cite{Ber09,BP08}, we can prove that $(p_\star (m K_{X/Y} + L), h)$ is positively curved on $Y_0$.
By studying the comportment of $h$ near $Y_1\setminus Y_0$, \cite{PT14} proved finally that $h$ can be extended 
as a possibly singular hermitian metric on $Y_1$ with positive curvature in the sense of Definition \ref{singularmetric}.
\end{remark}

The following proposition comes from the standard extension theorem.

\begin{proposition}\label{specialOT}
In the setting of Theorem \ref{maintool}, 
we suppose moreover that there exists a fibration $q : Y\rightarrow Z$ to some projective manifold 
$Z$. Let $H$ be a pseudo-effective line bundle on $Y$ with a possible singular metric $h_H$ such that $i\Theta_{h_H} (H) \geq 0$
in the sense of current. 
Let $A_Z$ be an ample line bundle on $Z$.
Then for $c\in\bN$ large enough (depending only on $A_Z$ and $Z$), the following extension property holds:
\smallskip

Let $z\in Z$ be a generic point and let $X_z$ (resp. $Y_z$) be the fiber of $p\circ q$ (resp. $q$) over $z$.
Let $e\in \mathcal{O}_{Z, z}(c\cdot A_Z)$ and
let $s\in H^0 (Y_z , K_Y \otimes H \otimes p_\star (m K_{X/Y} +L))$ such that 
\begin{equation}\label{l2condit}
\int_{Y_z} |s|^2 _{h_H, h} < + \infty , 
\end{equation}
where $h$ is the metric on $p_\star (m K_{X/Y} +L)$ in Theorem \ref{maintool} \footnote{As $q^\star K_Z$ is a trivial bundle on $Y_z$,
modulo this trivial bundle, $|s|^2 _{h_H, h}$ can be seen as a volume form on $Y_z$. Therefore the integral \eqref{l2condit} is well-defined.}.
Then there exists a section 
$$S\in H^0 (Y , K_Y \otimes H \otimes p_\star (m K_{X/Y} +L)\otimes q^\star (c A_Z))$$
such that $S |_{Y_z} = s \otimes  q^\star e$.
\end{proposition}

\begin{proof}
Let $(L_1, h_1)$ be the line bundle constructed in Remark \ref{hermtimetricconstr}. Then $s$ induces a section $u\in H^0 (X_z, K_X +H +L_1)$
and \eqref{l2condit} implies that
$$\int_{X_z} |s|^2 _{h_H, h_1} < +\infty .$$
For $c \in\mathbb{N}$ large enough (depending only on $Z$ and $A_Z$),
by the standard Ohsawa-Takegoshi extension theorem (cf. for example \cite[Chapter 13]{Dem12}), we can find a 
$$U\in H^0 (X, K_X +H+L_1 + (p\circ q)^\star c A_Z)$$
such that $U |_{X_z} = u \otimes (p\circ q)^\star e$. Then $U$ induces a section 
$$S\in H^0 (Y , K_Y \otimes H \otimes p_\star (m K_{X/Y} +L)\otimes q^\star (c A_Z))$$ 
such that $S |_{Y_z} = s \otimes  q^\star e$ and the proposition is proved.
\end{proof}

As another direct consequence of the Ohsawa-Takegoshi extension, the following proposition is will be important for us.
\begin{proposition}\cite[A.2]{BP10}\label{uniformbound}
In the setting of Theorem \ref{maintool}, let $U$ be a small Stein open subset of $X$ and let $V\Subset U$ be some open set of compact support in $U$. 
Let $e$ be a basis of $m K_{X/Y} +L$ over $U$. Then there exists a uniform constant $C (U, V, e)$ depending only on $U, V, e$ such that
for every $t\in \pi (V)$ and every $s\in H^0 (X_t , m K_{X/Y}+L)$, we have
$$\|\frac{s}{e} \|_{C^0 (V \cap \pi^{-1} (t))} \leq C (U,V,e) \cdot \|s\|_h .$$
\end{proposition}

\begin{proof}
As explained in Remark \ref{hermtimetricconstr}, the line bundle $L_1 := (m-1) K_{X/Y}+L$ can be equipped with a possibly singular metric $h_{L_1}$ such that 
$$i\Theta_{h_{L_1}} (L_1) \geq 0\qquad\text{on } X .$$
Since $U$ is a small open set, we can find a Stein open set $B \subset Y$ such that $U\subset p^{-1} (B)$.
As $m K_{X/Y} +L = K_{X/Y}+L_1$, by applying the Ohsawa-Takegoshi extension theorem to the fibration
$p^{-1} (B) \rightarrow B$, we can find a $\widetilde{s} \in H^0 (p^{-1}(B), K_X +L_1)$
such that 
\begin{equation}\label{OText}
\int_{p^{-1} (B)} |\widetilde{s}|_{h_{L_1}} ^2 \leq C \int_{X_t} |s|_{h_{L_1}} ^2  =C \cdot \|s\|_h ^2 
\end{equation}
and 
\begin{equation}\label{1extension}
\widetilde{s} |_{X_t} = s \wedge p^\star (e_{B})  ,
\end{equation}
where $e_B$ is a basis of $K_Y$ over $B$.

\medskip

On the open set $U$, $\widetilde{s}$ can be written as $\widetilde{s}= \widetilde{w} \cdot e \wedge p^\star (e_B)$ 
for some holomorphic function $\widetilde{w} $ on $U$.
Note that $V \Subset p^{-1} (B)$, \eqref{OText} implies thus that
$$\|\widetilde{w}\|_{C^0 (V)} \leq C (U, V, e) \cdot \|s\|_h $$
for some constant $C (U,V,e)$ depending only on $U, V$ and $e$. 
Thanks to \eqref{1extension}, we have $\widetilde{w} |_{X_t} = \frac{s}{e}$. Therefore
\begin{equation}
\|\frac{s}{e}\|_{C^0 (V \cap X_t)} \leq C (U, V, e) \cdot \|s\|_h .
\end{equation} 
The proposition is proved.
\end{proof}

\noindent The last result of this section concerns the regularity of the metric $h$.

\begin{proposition}\cite[Cor 2.8]{CP17}\label{maintool2}
Let $E \rightarrow X$ be a holomorphic vector bundle on a manifold $X$ (which is not necessary compact).
Let $h_E$ be a possibly singular hermitian metric on $E$ such that $(E, h_E)$ is positively curved.
Let $U$ be a topological open set of $X$.
If 
$$i\Theta_{\det h_E} (\det E) \equiv 0 \qquad\text{on }U,$$ 
then $h_E$ is a smooth metric on $E |_U $, and $(E |_U , h_E)$ is hermitian flat. 
\end{proposition}

\section{Proof of the main theorem}

\noindent We now prove the main theorem of the article.

\begin{theorem}[=Theorem \ref{main}]
Let $p: X\rightarrow Y$ be a fibration between two projective manifolds. Let $F$ be the generic fiber and let $\Delta$ be a $\bQ$-effective klt divisor on $X$. 
Set $\Delta_F := \Delta |_F$.
If $\dim Y \leq 2$, then 
\begin{equation}\label{iitakaineq}
\kappa (K_X +\Delta) \geq \kappa (K_F +\Delta_F) +\kappa (Y) . 
\end{equation}
\end{theorem}

\begin{proof}
Since $Y$ is of dimension $2$, we can consider its minimal model and assume that $Y$ is a smooth projective surface with nef canonical bundle. We show next that it will be enough to treat the case where $Y$ is a K3 surface.

Indeed, if $\kappa (Y) \geq 1$, as the klt version of $C_{n,1}$ is known (cf. \cite{Kaw82a, CP17}), 
we have thus \eqref{iitakaineq}. We refer to Proposition \ref{detailproof} in the appendix for a detailed proof.
If $\kappa (Y)=0$, by using the classification of minimal surface \cite[Thm 1.1]{BHPV}, 
we have $c_1 (Y)=0 \in H^{1,1} _{\bQ} (Y)$.
After a finite \'{e}tale cover\footnote{We remark that \eqref{iitakaineq} is invariant under this operation.}, 
the base $Y$ is either a torus or a K3 surface. If $Y$ is a torus, \cite[Thm 1.1]{CP17} implies \eqref{iitakaineq}.  
We assume in this way for the rest of our proof that \textit{$Y$ is a K3 surface}.  

\medskip

Let $m\in \bN$ be sufficiently divisible and let $Y_1$ be the locally free locus of the direct image sheaf $p_\star (m K_{X/Y}+m\Delta)$.
By using Theorem \ref{maintool}, there exists a possibly singular hermitian metric $h$ on 
$(p_\star (m K_{X/Y}+m\Delta), Y_1)$ such that 
\begin{equation}\label{griffithssemipos}
i\Theta_h (p_\star (m K_{X/Y}+m\Delta)) \succeq 0 \qquad\text{on }Y_1 ,
\end{equation}
and $h$ induces a hermitian metric $\det h$ on $(\det p_\star (m K_{X/Y}+m\Delta), Y)$ such that 
$$ i\Theta_{\det h} (\det p_\star (m K_{X/Y}+m\Delta)) \geq 0 \qquad\text{on }Y$$
in the sense of current. In particular, the bundle $\det p_\star (m K_{X/Y} +m \Delta)$ is pseudo-effective.

By Proposition \ref{keyprop}, we have a Zariski decomposition
\begin{equation}\label{Zariskidecomp}
\Det p_\star (m K_{X/Y}+m\Delta) \equiv_\bQ \sum_{i=1}^s  a_i [C_i] +L_m , 
\end{equation}
where $a_i \in \bQ^+$, $[C_i]$ are negative intersection curves, $L_m$ is nef and $L_m \cdot C_i =0$ for every $i$.
Let $\nd (L_m)$ be the numerical dimension of $L_m$. We will distinguish next among three cases, according to the 
numerical dimension of $L_m$.

\bigskip

{\em Case 1: The numerical dimension of $L_m$ equals two} 

Then we infer that the bundle $\det p_\star (m K_{X/Y}+m\Delta) $ is big on $Y$, and \eqref{iitakaineq} is thus proved by using \cite{Cam04b} or \cite[Thm 5.1]{CP17}.  

\bigskip

{\em Case 2: The numerical dimension of $L_m$ equals one}

Thanks to Remark \ref{semiample}, $L_m$ is semiample. Then $L_m$ induces a fibration $\pi : Y\rightarrow \bP^1$.
As $L_m \cdot [C_i] =0$ for every $i$, we have
\begin{equation}\label{inter0}
L_m \cdot  \det p_\star (m K_{X/Y}+m\Delta) =0 .
\end{equation}
By using \cite[Lemma 7.3]{Vie83}, 
we can find a birational morphism $Y'\rightarrow Y$ from a projective manifold $Y'$, and a desingularisation $X'$ of $Y'\times_Y X$ satisfying:
\begin{equation*} 
\begin{CD}
    X^\prime @>{\pi_X}>> X \\
    @Vp^\prime VV      @VVpV  \\ 
    Y^\prime @>>{\pi_Y}>   Y\\
          @.  @VV{\pi}V \\
          @.  \bP^1 
\end{CD}
\end{equation*} 
each divisor $W\subset X'$ such that $\codim_{Y'} p' (W) \geq 2$ is $\pi_X$-contractible.
Since $\Delta$ is klt, we can find a klt $\bQ$-effective divisor $\Delta'$ on $X^\prime$ and
some effective $\pi_X$-exceptional divisor $D'$ such that
\begin{equation}\label{klttransform}
\pi_X ^\star (K_X +\Delta) + D' =K_{X^\prime} + \Delta^\prime . 
\end{equation}
\smallskip

\noindent {\bf Claim.} The bundle 
$$\Det p' _\star (m K_{X'/Y'} +m \Delta') - c(\pi_Y \circ\pi) ^\star \mathcal{O}_{\bP^1} (1)$$ 
is pseudo-effective on $Y'$
for some constant $c>0$.  
\smallskip

\noindent We will verify this claim later; for now we finish the proof of the theorem.
By using \cite[Thm 3.4]{CP17}, the claim implies the existence of a divisor $E\subset X'$ 
such that $\codim_{Y'} p' (E) \geq 2$ and  
\begin{equation}\label{pseudo-effective}
D:= K_{X'/Y'}+\Delta' +E - \epsilon (p'\circ \pi_Y \circ \pi) ^\star \mathcal{O}_{\bP^1} (1)   
\end{equation}
is $\mathbb{Q}$-pseudo-effective on $X'$ for some $\epsilon >0$.

\medskip

Let $m_1 \gg m_2 \gg 1$. Thanks to \eqref{pseudo-effective}, we have
\begin{equation}\label{keydecom}
(m_1 +m_2) (K_{X'/Y'}+\Delta' +E )
\end{equation}
$$ = m_1 (K_{X'/Y'}+\Delta' + \frac{m_2}{m_1} D +E ) + \ep m_2 (p'\circ \pi_Y \circ \pi) ^\star \mathcal{O}_{\bP^1} (1) . $$
As $m_1 \gg m_2$, we can apply Theorem \ref{maintool} to $m_1 (K_{X'/Y'}+\Delta' + \frac{m_2}{m_1} D +E )$.
In particular, we can find a possibly singular metric $h_{m_1}$ 
on 
$$\sV_1 :=p' _\star (m_1 (K_{X'/Y'}+\Delta' + \frac{m_2}{m_1} D +E ) )$$ 
such that $i\Theta_{h_{m_1}} (\sV_1) \succeq 0$.
Set $T := i\Theta_{\det h_{m_1}} (\det \sV_1)$. Then $T \geq 0$ in the sense of current.
Let $Y' _t$ be a generic fiber of $\pi_Y \circ \pi$.

\medskip

If $T |_{Y' _t}$ is not identically $0$, as $Y' _t$ is of dimension $1$, $T |_{Y' _t}$ is strictly positive at a generic point of $Y' _t$.
Together this with \eqref{keydecom}, $\det p' _\star ((m_1 +m_2) (K_{X'/Y'}+\Delta' +E ))$ is big on $Y'$.
By applying \cite{CP17}, we get
\begin{equation}\label{withe1}
\kappa (X' , K_{X'}+\Delta' +E) \geq \kappa (F, K_F +\Delta_F) +2 . 
\end{equation}
As $E$ and $D'$ are $\pi_X$-contractible, \eqref{klttransform} and \eqref{withe1} imply \eqref{iitakaineq}.

\smallskip

If $T |_{Y' _t} \equiv 0$, thanks to Proposition \ref{maintool2}, 
$(\sV_1 |_{Y'_t} , h_{m_1})$ is hermitian flat on $Y' _t$.
In particular, $h_{m_1} |_{Y' _t}$ is a smooth metric.
Note that $H^0 (Y',  K_{Y'})$ is of dimension $1$. It defines a canonical metric $h_{Y'}$ on $K_{Y'}$ and the restriction of $h_{Y'}$
on $Y' _t$ is smooth.
As a consequence, we have
$$
\int_{Y' _t}  |s|_{(m_1 -1) h_{Y'} , h_{m_1}} ^2 < +\infty \qquad\text{for every } s\in H^0 (Y' _t,  K_{Y'}\otimes (m_1 -1) K_{Y'} \otimes\sV_1).
$$
Combining this with Proposition \ref{specialOT} \footnote{We take $H = (m_1 -1) K_{Y'}$ and $h_H = (m_1 -1) h_{Y'}$.}, we get
$$h^0 (Y',  K_{Y'} \otimes (m_1 -1) K_{Y'} \otimes \sV_1 \otimes \ep m_2 (\pi_Y \circ \pi) ^\star \mathcal{O}_{\bP^1} (1))$$
$$\geq h^0 (Y' _t, 
 K_{Y'} \otimes (m_1 -1) K_{Y'} \otimes \sV_1) = h^0(X' _t,m_1 (K_{X' _t}+\Delta' + \frac{m_2}{m_1} D +E )) .$$
Together with \eqref{keydecom}, we obtain 
\begin{equation}\label{keyineq}
h^0 (X', 
m_1 \cdot (p')^\star K_{Y'} +(m_1 + m_2) (K_{X'/Y'}+\Delta' +E)) \geq  h^0(X' _t,m_1 (K_{X' _t}+\Delta' + \frac{m_2}{m_1} D +E )) .
\end{equation}

\medskip

Finaly, by applying \cite{Kaw82a,CP17} to $X'_t \rightarrow Y' _t$, we have 
$$\kappa (X' _t, K_{X' _t}+\Delta' + \frac{m_2}{m_1} D + E) \geq \kappa (F, K_F +\Delta_F) .$$
Together with \eqref{keyineq} and the fact that $K_{Y'}$ is $\mathbb{Q}$-effective, we obtain
\begin{equation}\label{withE}
\kappa (X' , K_{X'}+\Delta' + E) \geq \kappa (F ,K_F +\Delta_F) .
\end{equation}
As $E$ and $D'$ are $\pi_X$-contractible, \eqref{klttransform} and \eqref{withE} imply \eqref{iitakaineq}.

\bigskip

{\em Case 3: The numerical dimension of $L_m$ equals zero}

Then $L_m$ is trivial (as it is semiample) and we have
\begin{equation}\label{effectiveness}
\Det p_\star (m K_{X/Y}) \equiv_\bQ \sum_{i=1}^s  a_i [C_i]  
\end{equation}
where $[C_i]$ are negative curves. As $ i\Theta_{\det h} (\det p_\star (m K_{X/Y}+m\Delta))$ is a positive current in the same class of
$\sum\limits_{i=1}^s  a_i [C_i] $, we get
$$i\Theta_{\det h} (\det p_\star (m K_{X/Y}+m\Delta)) = \sum_{i=1}^s  a_i [C_i]  \qquad\text{on }Y $$
in the sense of current. In particular, we have
$$i\Theta_{\det h} (\det p_\star (m K_{X/Y}+m\Delta)) \equiv 0 \qquad\text{on }Y \setminus (\cup C_i) .$$
By using Proposition \ref{maintool2}, $(p_\star (m K_{X/Y}+m\Delta), h)$ is hermitian flat on $Y_1 \setminus (\cup C_i)$.
\medskip

Let $:\tau: Y\rightarrow Y_{\can}$ be the morphism which contracts the negative curves $\cup C_i$. 
There are two possible cases: 
$\pi_1 (Y\setminus (\cup C_i))$ is finite or infinite. We will analyze each possibility.

\subsubsection{The fundamental group $\pi_1 \big(Y\setminus (\cup C_i)\big)$ is finite.} As $\codim_Y (Y\setminus Y_1) \geq 2$, 
we know that $\pi_1 (Y_1\setminus (\cup C_i)) =\pi_1 (Y \setminus (\cup C_i))$ is finite.
Let $r$ be the number of elements of the finite group $\pi_1 (Y_1\setminus (\cup C_i))$.
Fix a generic point $y\in Y_1\setminus (\cup C_i)$.
As the direct image vector bundle $(p_\star (m K_{X/Y}+m\Delta), h)$ is hermitian flat on $Y_1\setminus (\cup C_i)$, the parallel transport induces a representation
\begin{equation}
\rho: \pi_1 (Y_1 \setminus (\cup C_i)) \rightarrow  \Aut (H^0 (X_y , m K_{X/Y}+m \Delta)) .
\end{equation}
Let $f\in H^0 (X_y , m K_{X/Y}+m \Delta)$ be an element with unit norm. Although the parallel transport of $f$ cannot induce a global section over $Y_1\setminus (\cup C_i)$, 
the corresponding parallel transport of 
$$\prod_{a \in\pi_1 (Y_1\setminus (\cup C_i))} \rho (a) (f) \in H^0 (X_y ,  mr (K_{X/Y}+ \Delta))$$
induces a section $\widetilde{f} \in H^0 (p^{-1}(Y_1 \setminus (\cup C_i)), mr (K_{X/Y}+\Delta))$.

\medskip

We now prove that $\widetilde{f}$ can be extended to the total space $X$.
Let $U$ be an arbitrary small Stein open subset of $X$ and $V\Subset U$ be some arbitrary open set with compact support in $U$. 
Let $e$ be  a basis of $m K_{X/Y}+ m \Delta$ on $U$.
We have 
$\widetilde{f} = \widetilde{l} \cdot e^{\otimes r}$ for some holomorphic function 
$$\widetilde{l} \in H^0 (V \cap p^{-1} (Y_1\setminus (\cup C_i)), \mathcal{O}_{V \cap p^{-1} (Y_1\setminus (\cup C_i))}) .$$
By construction, on every fiber $X_t$, $\widetilde{f} =\prod_{i=1}^r f_i$ for some $f_i \in H^0 (X_t,  m K_{X/Y}+m \Delta)$ with unit norm.
Thanks to Proposition \ref{uniformbound}, the $C^0$-norm $\|\frac{f_i}{e} \|_{C^0 (V \cap X_t)}$ is bounded by a constant $C (U,V,e)$ independent of $t$. 
Therefore 
$$\|\widetilde{l}\|_{C^0 (V \cap X_t)} = \|\prod_{i=1}^r \frac{f_i}{e} \|_{C^0 (V \cap X_t)}\leq C (U, V,e) ^r .$$ 
In particular, $|\widetilde{l}|$ is bounded on $V \cap p^{-1} (Y_1\setminus (\cup C_i))$ and
$\widetilde{f}$ can be thus extended  as a holomorphic section on $V$.
Since $V$ is an arbitrary small open set in $X$, 
$\widetilde{f}$ can be extended to the total space $X$.

\medskip

In conclusion, for any element $f\in H^0 (X_y , m K_{X/Y}+m \Delta)$, we can find a 
$$\widetilde{f} \in H^0 (X , mr (K_{X/Y}+\Delta))$$
such that $\widetilde{f}  |_{X_y} = \prod_{a \in\pi_1 (Y_1\setminus (\cup C_i))} \rho (a) (f)$.
In particular, we have 
$$\Div (\widetilde{f}  |_{X_y}) = \sum_{a \in\pi_1 (Y_1\setminus (\cup C_i))} \Div (\rho (a) (f)) .$$
Therefore, $\kappa (K_X +\Delta) \geq 1$ if $\kappa (K_F +\Delta_F) \geq 1$. In other words, we have 
$$\kappa (K_X +\Delta) \geq \min \{1, \kappa (K_F +\Delta_F) \} .$$
Together with a standard argument (cf. Proposition \ref{simplecmn} in the appendix), we get
$$\kappa (K_X +\Delta) \geq \kappa (K_F +\Delta_F)$$
and the first subcase is completely proved.

\medskip

\subsubsection{The fundamental group $\pi_1 \big(Y\setminus (\cup C_i)\big)$ is not finite.} As a consequence of Proposition \ref{uniformation}, there exists a orbifold cover from a complex torus $T$ to $Y_{\can}$:
$$\tau_Y : T \rightarrow Y_{\can} .$$
Let $X'$ be a desingularisation of $X\times_{Y_{\can}} T$. We have thus a commutative diagram
\begin{equation*} 
\begin{CD}
    X^\prime @>\tau_X >> X \\
    @Vp^\prime VV      @VVpV  \\ 
    T @>>\tau_Y >   Y_{\can}\\
\end{CD}
\end{equation*} 
Set $T_1 := \tau_Y ^{-1} (\tau (Y_1\setminus (\cup C_i) ))$, where $\tau: Y\rightarrow Y_{\can}$ is the contraction morphism.
Thanks to Proposition \ref{uniformation}, $\tau_Y$ is a non-ramified cover on $T_1$ and 
$$\codim_{T} (T\setminus T_1) \geq 2 .$$
As $\Delta$ is klt, we can find a klt $\bQ$-effective divisor $\Delta'$ on $X^\prime$ and some $\bQ$-divisor $D'$ supported in $(p')^{-1} (T\setminus T_1)$ such that
\begin{equation}\label{klttransform1}
\pi_X ^\star (K_X +\Delta) + D' =K_{X^\prime} + \Delta^\prime .
\end{equation}
Since $T$ is a torus, by applying \cite{CP17}, we have
$$\kappa (K_{X'} +\Delta') \geq \kappa (K_F +\Delta_F) .$$

Let $m\in\mathbb{N}$ be a sufficiently divisible number and let $s\in H^0 (X' , m K_{X' /T} + m \Delta')$. 
Thanks to \eqref{klttransform1} and the fact that $D'$ is supported in $(p')^{-1} (T\setminus T_1)$,
$s$ induces an element 
$$s_T\in H^0 (T_1, \tau_Y ^\star (p_\star (m K_{X/Y}+m\Delta))) .$$
Since $( p_\star (m K_{X/Y} +m\Delta), h)$ is hermitian flat on $Y_1$,
$\|s_T\|^2 _{(\pi_Y)^\star h} (t)$ is a psh function on $t\in T_1$.
As $\codim_{T} (T\setminus T_1) \geq 2$, $\|s_T\|_{\pi_Y ^\star h} (t)$ is thus constant with respect to $t\in T_1$.
Let $r$ be the degree of the cover $\tau_Y$. Since $\tau_Y$ is a non-ramified cover on $T_1$,
$s_T$ induces an element $\widetilde{s}\in H^0 (p^{-1}(Y_1 \setminus (\cup C_i)), m r K_{X/Y}+m r \Delta)$. 
As $\|s_T\|_{\tau_Y ^\star h} (t)$ is constant, by using the same argument as in the subcase 3.0.1, $\widetilde{s}$ can be extended to as an element in
$H^0 (X, m r K_{X/Y}+m r \Delta)$.
\eqref{iitakaineq} is thus proved by using the same argument as in the end of Subcase 3.0.1.
\end{proof}

\bigskip

\noindent Our next job is to establish the claim used in the proof of our main result, 
which is a consequence of the volume estimate inequality (or the holomorphic Morse inequalities).

\begin{proof}[Proof of the claim]
Thanks to \cite{PT14}, we know that $\Det p' _\star (m K_{X'/Y'} +m \Delta')$ is pseudo-effective on $Y'$.
Let $A$ be the nef part of the Zariski decomposition of $\Det p' _\star (m K_{X'/Y'} +m \Delta')$.
Set $B:= (\pi_Y \circ\pi) ^\star \mathcal{O}_{\bP^1} (1)$.
As $B$ is semiample, we have
$$ 0\leq A\cdot B \leq c_1 (\Det p' _\star (m K_{X'/Y'} +m \Delta')) \cdot c_1 (B) $$
$$=(\pi_Y)_\star (c_1 (\Det p' _\star (m K_{X'/Y'} +m \Delta'))) \cdot \pi^\star c_1 (\mathcal{O}_{\bP^1} (1))$$
$$= c_1 (\Det p_\star (m K_{X/Y} +m \Delta)) \cdot \pi^\star c_1 (\mathcal{O}_{\bP^1} (1)) =0 ,$$
where the last equality a consequence of \eqref{inter0}. Then we have
\begin{equation}\label{nulintersection}
A \cdot B  =0 .
\end{equation}

Let $L$ be an ample line bundle on $Y'$ and set $c := \frac{L\cdot A}{2 L \cdot B} \in \bQ^{+}$. 
For any $\tau \in \bQ^+$ small enough, thanks to \eqref{nulintersection} and the choice of $c$, 
the basic volume estimate (cf. for example \cite[8.4]{Dem12} or \cite[Thm 2.2.15]{Laz04}) implies that
$$\vol (A +\tau L - c B) \geq (A +\tau L)^2 -2c (A +\tau L)\cdot B $$
$$\geq 2\tau (L\cdot A - c L \cdot B) + o(\tau) > 0 .$$
Therefore $A +\tau L - c B$ is big for any $\tau \in \bQ^+$. Letting $\tau\rightarrow 0^+$, 
$A-c B$ is pseudoeffective. Then $\Det p' _\star (m K_{X'/Y'} +m \Delta') - c B$ is pseudo-effective and the claim is proved.
\end{proof}

\section{Appendix}

\noindent In this appendix, we will gather two standard results which should be well-known to experts.
\begin{proposition}\label{simplecmn}\cite{Kaw82a, CH11}
Let $p: X\rightarrow Y$ be a fibration from a $n$-dimensional projective manifold to a K3 surface, and let $\Delta$ be an effective klt $\bQ$-divisor on $X$.
Assume that Theorem \ref{main} holds for $\dim X \leq n-1$. 
If $\kappa (K_X +\Delta) \geq 1$, 
then 
$$\kappa (K_X +\Delta) \geq \kappa (K_F +\Delta_F) ,$$
where $F$ is the generic fiber of $p$ and $\Delta_F = \Delta |_F$.
\end{proposition}

\begin{proof}
We use here the argument in \cite[Prop 3.7]{CP17}. Modulo desingularization, 
we can assume that the Iitaka fibration of $K_X+\Delta$ is a morphism between two projective manifolds
$\varphi: X\rightarrow W$. 
$$\xymatrix{
X\ar[rr]^-{\varphi} \ar[rd]_{p}
& & W \\
& Y }
$$
Let $G$ be the generic fiber of $\varphi$ and set $\Delta_G :=\Delta |_G$. 
Then 
\begin{equation}\label{genericfiber}
\kappa (K_G +\Delta_G)=0 . 
\end{equation}
Let $p: G \rightarrow p (G)$ be the restriction of $p$ on $G$. We will analyze next among three cases
which may occur. 

\medskip

\emph{Case 1: We assume that $p(G)$ projects onto $Y$}; then we argue as follows.
Let $\widetilde{p} : G \rightarrow \widetilde{Y}$ be the Stein factorization of $p : G \rightarrow Y$:
$$\xymatrix{
G\ar[rr]^-{p} \ar[rd]_{\widetilde{p}}
& & Y \\
& \widetilde{Y} \ar[ru]_s}
$$
After desingularization $\widetilde{p}$, we can assume that $\widetilde{Y}$ is smooth. Let $G_t$ be the generic fiber of $\widetilde{p}$.
By assumption, Theorem \ref{main} holds for $G\rightarrow \widetilde{Y}$.
Therefore \eqref{genericfiber} implies that 
\begin{equation}\label{kodairadim0}
 \kappa (K_{G_t} +\Delta_{G_t})=0 .
\end{equation}

We estimate next the dimension of $G$. Let $F$ be the generic fiber of $p : X\rightarrow Y$.
By restricting $\varphi$ on $F$, we obtain a morphism
$$\varphi_t: F \rightarrow V$$
where $V$ is a subvariety of $W$.
Let $\widetilde{V} \rightarrow V$ be the Stein factorization of $\varphi_t$.
$$\xymatrix{
F\ar[rr]^-{\varphi_t} \ar[rd]_{\widetilde{\varphi}_t}
& & V\\
& \widetilde{V} \ar[ru] }
$$
Since $G$ is generic, we infer that the generic fiber of $\widetilde{p}$ coincides with the generic fiber of $\widetilde{\varphi}_t$.
Combining this with \eqref{kodairadim0}, then \cite[Thm 5.11]{Uen75} implies that 
$$\kappa (K_F+\Delta_F) \leq \dim \widetilde{V} = \dim F -\dim G_t .$$
Therefore we have
$$\dim G_t \leq \dim F-\kappa (K_F+\Delta_F)$$
and thus we infer that
$$\dim G = \dim G_t +\dim \widetilde{Y} \leq \dim F -\kappa (K_F +\Delta_F) +\dim Y = \dim X -\kappa (K_F +\Delta_F) .$$

Finally, by construction of the Iitaka fibration, $\dim G = \dim X -\kappa (K_X+\Delta)$;
we obtain the inequality
$$\dim X -\kappa (K_X+\Delta) \leq \dim X -\kappa (K_F +\Delta_F),$$
and in conclusion $\kappa (K_X+\Delta) \geq \kappa (K_F +\Delta_F)$.

\medskip

\emph{Case 2: We assume that the image $p(G)$ has dimension zero.}
Since $G$ is connected, $p (G)$ is a point in $Y$.
This means that we can define a map $W\to Y$, which can be assumed to be regular by blowing up $W$. We have thus the commutative diagram
$$\xymatrix{
X\ar[rr]^-{\varphi} \ar[rd]_{p}
& & W\ar[ld]^q \\
& Y }
$$
Set $t:= p (G)$.
Let $F$ be the fiber of $p$ over $t$. Then $F$ is a generic fiber of $p$ and $G$ is a generic fiber of 
$$\varphi: F \rightarrow \varphi (F),$$
and by \cite[Thm 5.11]{Uen75} we infer that 
$$\kappa (K_F +\Delta_F) \leq \kappa (K_G +\Delta_G) + \dim \varphi (F) = \dim \varphi (F).$$
Note that $\varphi (F)$ is the fiber of $q$ over $t\in Y$.
We have $\dim W = \dim \varphi (F) + \dim Y$.
Therefore $\dim W \geq  \kappa (K_F +\Delta_F) +\dim Y$.
Combining this with the fact that $\varphi$ is the Iitaka fibration, we have thus 
$$\kappa (K_X +\Delta)=\dim W \geq \kappa (K_F +\Delta_F) +\dim Y ,$$
and we are done.

\medskip

\emph{Case 3: The remaining case: $p (G)$ is a proper subvariety of $Y$}.

Let $p (G) '$ be the normalization of $p (G)$. If $p (G) '$ is a curve of general type, 
then $\kappa (K_G +\Delta_G) \geq 1$ and we get a contradiction with \eqref{genericfiber}.
If $p (G) '$ is $\bP^1$, as $G$ is generic, $Y$ is thus covered by rational curves. 
We get a contradiction with the assumption that $Y$ is K3.
As a consequence, $p (G) '$ is a torus. Then $[p (G)]$ is a semi-ample class of numerical dimension $1$ in $Y$.
Therefore $p (G)$ is a generic fiber of a fibration $\pi: Y\rightarrow \bP^1$.
We have thus the following commutative diagram
$$\xymatrix{
X\ar[rr]^-{\varphi} \ar[rd]_{p\circ \pi}
& & W\ar[ld]^q \\
& \bP^1 }
$$

Set $t:= p\circ \pi (G)$.
Let $X_t$ be the fiber of $p\circ \pi$ over $t$. Then $G$ is the generic fiber of 
$$\varphi |_{X_t} : X_t \rightarrow \varphi (X_t),$$
and by \cite[Thm 5.11]{Uen75} we infer that 
$$\kappa (K_{X_t} +\Delta_{X_t}) \leq \kappa (K_G+\Delta_G) + \dim \varphi (X_t) = \dim \varphi (X_t).$$
Note that $\varphi (X_t)$ is the fiber of $q$ over $t\in \bP^1$.
We have $\dim W = \dim \varphi (X_t) + 1$.
Therefore $\dim W \geq  \kappa (K_{X_t} +\Delta_{X_t}) + 1$.
Combining this with the fact that $\varphi$ is the Iitaka fibration, we have thus 
$$\kappa (K_X +\Delta_X) =\dim W \geq \kappa (K_{X_t} +\Delta_{X_t}) + 1 \geq \kappa (K_F +\Delta_F)+1 ,$$
where the last inequality comes from the fact that $X_t$ is a fibration over a torus with the generic fiber $F$. 
The proposition is thus proved.
\end{proof}

\begin{proposition}\label{detailproof}
Let $p: X\rightarrow Y$ be a fibration between two projective manifolds. Let $F$ be the generic fiber and let $\Delta$ be a $\bQ$-effective klt divisor on $X$. 
Set $\Delta_F := \Delta |_F$.
If $\dim Y = 2$ and $\kappa (Y) \geq 1$, then 
\begin{equation}\label{iitakaineqapped}
\kappa (K_X +\Delta) \geq \kappa (K_F +\Delta_F) +\kappa (Y) . 
\end{equation}  
\end{proposition}

\begin{proof}
Since $Y$ is of dimension $2$, we can consider its minimal model and assume that $Y$ is smooth with semi-ample canonical bundle.

If $\kappa (Y)=2$, then $K_Y$ is big and it is known that \eqref{iitakaineqapped} holds.

If $\kappa (Y)=1$, we can suppose that $K_Y$ is semi-ample. Then $K_Y$ induces a fibration 
$$\pi: Y\rightarrow Z$$
to a $1$-dimensional variety $Z$ and $K_Y = \pi^\star A$ for some ample line bundle $A$ on $Z$.
Let $Y_z$ be a generic fiber of $\pi$. Then $Y_z$ is a $1$-torus.
Let $m\in\mathbb{N}$ be a number sufficiently large and let $h$ be the possibly singular hermitian metric on
$p_\star (m K_{X/Y} +m \Delta)$ defined in Theorem \ref{maintool}. There are two cases.

\medskip

{\em Case 1.} $ i\Theta_{\det h} (\det p_\star (m K_{X/Y} +m \Delta)) |_{Y_z} \equiv 0$. 
Thanks to Proposition \ref{maintool2}, the vector bundle 
$(p_\star (m K_{X/Y} +m \Delta)|_{Y_z} , h )$ is hermitian flat.
Therefore
\begin{equation}\label{l2cond}
\int_{Y_z} |s|_{h} ^2 < +\infty \qquad\text{for every }s\in H^0 (Y_z, p_\star (m K_{X/Y} +m \Delta)).
\end{equation}
As $K_Y =\pi^\star A$ for some ample line bundle on $Z$, Proposition \ref{specialOT} and the $L^2$-condition \eqref{l2cond} imply that
\begin{equation}\label{2ineq}
\kappa (X, K_X +\Delta) \geq \kappa (X_z, K_{X/Y} +\Delta |_{X_z}) +1.
\end{equation}
Moreover, by applying \cite{CP17, Kaw82a} to $X_z \rightarrow Y_z$ and the fact that $Y_z$ is a torus, we have
\begin{equation}
\kappa (X_z, K_{X/Y} +\Delta |_{X_z})  \geq \kappa (K_F +\Delta_F) .
\end{equation}
Together with \eqref{2ineq}, \eqref{iitakaineqapped} is proved.

\medskip

{\em Case 2.} $ i\Theta_{\det h} (\det p_\star (m K_{X/Y} +m \Delta)) |_{Y_z} \gneqq  0$. As $Y_z$ is of dimension $1$, 
$\det p_\star (m K_{X/Y} +m \Delta) |_{Y_z}$ is ample on $Y_z$. 
Since $K_Y$ is semi-ample, we can find some $\mathbb{Q}$-div $\Delta' \geq 0$ 
in the same class 
of $c\cdot p^\star K_Y$ for some $c >0$ small enough such that $\Delta +\Delta'$ is klt. 
Then 
$$\det p_\star (m K_{X/Y} +m \Delta +m \Delta') = \det p_\star (m K_{X/Y} +m \Delta) + m\Delta'$$ 
is big on $Y$.
By applying for example \cite{Cam04b, CP17}, we have
$$\kappa (K_{X/Y} + \Delta + \Delta') \geq \kappa (K_F +\Delta_F) +2 .$$
As $c<1$, we know that $K_X + \Delta - (K_{X/Y} + \Delta + \Delta') = (1-c) \cdot p^\star K_Y$ is $\mathbb{Q}$-effective.
Therefore
$$\kappa (K_X + \Delta) \geq \kappa (K_F +\Delta_F) +2 ,$$
and \eqref{iitakaineqapped} is proved.
\end{proof}


\begin{thebibliography}{666666}

\bibitem[BHPV]{BHPV} 
Barth, Wolf P. and Hulek, Klaus and Peters, Chris A. M. and Van de Ven, Antonius: \emph{Compact Complex surfaces}  Second edition. 
Ergebnisse der Mathematik und ihrer Grenzgebiete. 3. Folge. A Series of Modern Surveys in Mathematics, 4. Springer-Verlag, Berlin,

\bibitem[Ber09]{Ber09} 
Berndtsson, Bo:
\emph{Curvature of vector bundles associated to holomorphic fibrations.}
Ann. of Math. (2) 169 (2009), no. 2, 531–-560. 

\bibitem[BP08]{BP08} 
Berndtsson, Bo and P\u{a}un, Mihai: 
\emph{Bergman kernels and the pseudoeffectivity of relative
canonical bundles.} Duke Math. J., 145(2):341–-378, 2008.


\bibitem[BP10]{BP10} 
Berndtsson, Bo and P\u{a}un, Mihai: 
\emph{Bergman kernels and subadjunction.} arxiv: 1002.4145v1

\bibitem[Bir09]{Bir09} 
Birkar, Caucher:
\emph{The Iitaka conjecture $C_{n,m}$ in dimension six.}
Compos. Math. 145 (2009), no. 6, 1442–-1446.

\bibitem[BC15]{BC15}
Birkar, Caucher and Chen, Jungkai Alfred:
\emph{Varieties fibred over abelian varieties with fibres of log general type.}
Adv. Math. 270 (2015), 206–-222.

\bibitem[Cam04a]{Cam04a}
Campana, Fr\'ed\'eric: \emph{Orbifoldes à première classe de Chern nulle} (French. English summary) 
[Orbifolds of zero first Chern class] The Fano Conference, 339–-351, Univ. Torino, Turin, 2004. 

\bibitem[Cam04b]{Cam04b}
Campana, Fr\'ed\'eric: \emph{Orbifolds, special varieties and classification theory.}  
Ann. Inst. Fourier (Grenoble) 54 (2004), no. 3, 499–-630. 

\bibitem[CC14]{CC14} Campana, Fr\'{e}d\'{e}ric and Beno\^{\i}t, Claudon:
\emph{Abelianity conjecture for special compact Kähler 3-folds.} (English summary)
Proc. Edinb. Math. Soc. (2) 57 (2014), no. 1, 55–-78. 

\bibitem[CP17]{CP17} 
Cao, Junyan and P\u{a}un, Mihai: 
\emph{Kodaira dimension of algebraic fiber spaces over Abelian varieties} Invent. Math. 207 (2017), no. 1, pp 345--387

\bibitem[CH11]{CH11} 
Chen, Jungkai Alfred and Hacon, Christopher D.: 
\emph{Kodaira dimension of irregular varieties} Invent. Math. 186, (2011), 481--500

\bibitem[Dem12]{Dem12}
Demailly, Jean-Pierre 
\emph{Analytic methods in algebraic geometry.} 
Surveys of Modern Mathematics, 1. International Press, Somerville, MA; Higher Education Press, Beijing, 2012. viii+231 pp.

\bibitem[Fuj16]{Fuj16}
Fujino, Osamu: 
\emph{Direct images of relative pluricanonical bundles} Algebr. Geom. 3 (2016), no. 1, 50–-62. 

\bibitem[Fuj78]{Fuj78}
Fujita, Takao: 
\emph{On K\"{a}hler fiber spaces over curves.} J. Math. Soc. Japan 30, 779--794 (1978)

\bibitem[Fuj79]{Fuj79} 
Fujita, Takao: 
\emph{On Zariski problem}
Proc. Japan Acad. Ser. A Math. Sci. 55 (1979), no. 3, 106–110. 

\bibitem[GKP15]{GKP15} Kebekus, Stefan; Greb, Daniel and Peternell, Thomas: 
\emph{Singular spaces with trivial canonical class}
To appear in: Minimal models and extremal rays — proceedings of the conference in honor of Shigefumi Mori’s 60th birthday, 
Advanced Studies in Pure Mathematics, Kinokuniya Publishing House, Tokyo.

\bibitem[Hor10]{Hor10} 
H\"{o}ring, Andreas: \emph{Positivity of direct image sheaves—a geometric point of view}
Enseign. Math. (2) 56 (2010), no. 1-2, 87–-142.

\bibitem[KP17]{KP17}
Kov\'{a}cs, S\'{a}ndor J. and Patakfalvi, Zsolt:
\emph{Projectivity of the moduli space of stable log-varieties and subadditivity of log-Kodaira dimension.}
J. Amer. Math. Soc. 30 (2017), no. 4, 959--1021.

\bibitem[Gri70]{Gri70}
Griffiths, Phillip A.: 
\emph{Periods of integrals on algebraic manifolds III.} 
Publ. Math. IHES, 38, 125--180 (1970)

\bibitem[HPS16]{HPS16}
Hacon, Christopher; Popa, Mihnea and Schnell, Christian:
\emph{Algebraic fiber spaces over abelian varieties: around a recent theorem by Cao and Paun}
arXiv:1611.08768
 
\bibitem[Kaw82a]{Kaw82a} 
Kawamata, Yujiro:
\emph{Kodaira dimension of algebraic fiber spaces over curves}
Invent. Math. 66 (1982), no. 1, 57–-71. 

\bibitem[Kaw82b]{Kaw82b}  
Kawamata, Yujiro 
\emph{Characterization of Abelian varieties} Compositio Mathematica  43 (1981), no. 2, 253–-276

\bibitem[Kol85]{Kol85}
Koll\'{a}r, J\'{a}nos
\emph{Subadditivity of the Kodaira dimension: fibers of general type} 
Algebraic geometry, Sendai, 1985, 361–-398, 
Adv. Stud. Pure Math., 10, North-Holland, Amsterdam, 1987. 


\bibitem[KM98]{KM98} 
Koll\'{a}r, J\'{a}nos and Mori, Shigefumi: 
\emph{Birational geometry of algebraic varieties} 
With the collaboration of C. H. Clemens and A. Corti. Translated from the 1998 Japanese original. Cambridge Tracts in Mathematics, 134. 
Cambridge University Press, Cambridge, 1998. viii+254 pp.

\bibitem[Laz04]{Laz04}
Lazarsfeld, Robert 
\emph{Positivity in algebraic geometry. I. Classical setting: line bundles and linear series.}
Ergebnisse der Mathematik und ihrer Grenzgebiete. 3. Folge. A Series of Modern Surveys in Mathematics [Results in Mathematics and Related Areas. 3rd Series. A Series of Modern Surveys in Mathematics], 48. Springer-Verlag, Berlin, 2004. 

\bibitem[PT14]{PT14}
P\u{a}un, Mihai and Takayama, Shigeharu: 
\emph{Positivity of twisted relative pluricanonical bundles and their direct images}
arxiv 1409.5504 

\bibitem[Pau16]{Pau16} P\u{a}un, Mihai: 
\emph{Singular Hermitian metrics and positivity of direct images of pluricanonical bundles}
arXiv:1606.00174


\bibitem[Rau15]{Rau15} Raufi, Hossein: 
\emph{Singular hermitian metrics on holomorphic vector bundles,} 
Arkiv f\"{o}r Matematik,  October 2015, Volume 53, Issue 2, pp 359--382

\bibitem[Uen75]{Uen75}
Ueno, Kenji: 
\emph{Classification theory of algebraic varieties and compact complex spaces,}
Notes written in collaboration with P. Cherenack. 
Lecture Notes in Mathematics, Vol. 439. Springer-Verlag, Berlin-New York, 1975. xix+278 pp.

\bibitem[Vie83]{Vie83} 
Viehweg, Eckart: 
\emph{Weak positivity and the additivity of the Kodaira dimension for certain fibre spaces,} 
Algebraic varieties and analytic varieties (Tokyo, 1981), 329–-353, Adv. Stud. Pure Math., 1, North-Holland, Amsterdam, 1983. 

\bibitem[Vie95]{Vie95} Viehweg, Eckart:
\emph{Quasi-projective moduli for polarized manifolds,}
Ergebnisse der Mathematik und ihrer Grenzgebiete (3) [Results in Mathematics and Related Areas (3)], 30. Springer-Verlag, Berlin, 1995. viii+320 pp.
\end{thebibliography}
\end{document}